%% file: Betti_Stabilization_of_Generic_Initial_System.tex
\documentclass[11pt]{amsart}
\usepackage[pdftex]{graphicx}
\usepackage{amsmath}
\usepackage{color}
\usepackage{amsthm}
\usepackage{multicol}

\usepackage{setspace} 

\usepackage{amssymb}

\theoremstyle{theorem}
\newtheorem{thm}{Theorem}[section]
\newtheorem{lem}[thm]{Lemma}

\newtheorem{cor}[thm]{Corollary}

\newtheorem{question}[thm]{Question}

\theoremstyle{definition}
\newtheorem{example}[thm]{Example}
\newtheorem{defn}[thm]{Definition}

\usepackage{enumitem}

\usepackage[margin=1.3in]{geometry}

\usepackage{url}

\author{Sarah Mayes-Tang}
\address{Quest University Canada, 3200 University Blvd, Squamish, British Columbia, V8B 0N8}

\begin{document}

\title{Asymptotic stabilization of Betti diagrams of generic initial systems}


\maketitle

\input{Abstract}

\input{Examples}

\input{Background}

\input{StabilizationProof}

\bibliography{BSBib}
\bibliographystyle{amsalpha}
\nocite{*}

\end{document}

%% file: Abstract.tex
\begin{abstract}
Several authors investigating the asymptotic behaviour of the Betti diagrams of the graded system $\{ I^k \}$ have shown that the shape of the nonzero entries in the diagrams stabilizes when $I$ is a homogeneous ideal with generators of the same degree.  In this paper, we study the Betti diagrams of graded systems of ideals built by taking the initial ideals or generic initial ideals of powers, and discuss the stabilization of additional collections of Betti diagrams.  Our main result shows that when $I$ has generators of the same degree, the entries in the Betti diagrams of the reverse lexicographic generic initial system $\{ \text{gin}_{\tau}(I^k) \}$ are given asymptotically by polynomials and that  the shape of the diagrams stabilizes.  

\end{abstract}

%% file: Examples.tex
\section{Betti Tables of Graded Systems of Ideals:  Examples}\label{sec:examples}

Cutkosky, Herzog, and Trung proved that when $I$ is an ideal in $S=K[x_1, \dots, x_n]$, the regularity of $I^k$ is a  linear function for $k$ large enough (\cite{CHT99}). Following this lead, later  work  examined the asymptotic behaviour of finer invariants of the set of powers of an ideal, $\{ I^k \}$, such as the the Betti numbers and graded Betti numbers $\beta_{i,j}(I) = \text{dim}_K \text{Tor}_i(K, I)_j$ (\cite{LV04}, \cite{Singla07}, \cite{BCH11}, \cite{Wh14}).  The goal of this paper is to examine the behaviour of the graded Betti numbers of other  families of ideals.  In particular, we will consider \textbf{graded systems of ideals} $\{ J_k \}$ satisfying  $J_i \cdot J_j \subseteq J_{i+j}$ for all $i$ and $j$.  

One type of end behaviour of graded Betti numbers of ideals in a graded system occurs when  the shape of the nonzero entries in the Betti diagrams \textit{stabilizes} in the following sense (see \cite{Wh14}).

\begin{defn}\label{defn:stabilize}
Consider a graded system of ideals $\{ J_k \}$.  We say that the \textbf{shape of $\beta(J_k)$ stabilizes} if there exist integers $k_0$ and $r$ such that for all $i$, $j$, and  $k\geq k_0$,
$$\beta_{i,kr+j}(J_k) \neq 0 \text{ if and only if } \beta_{i,k_0r+j}(J_{k_0}) \neq 0.$$ 
The smallest  $k_0$ satisfying this definition is called the \textbf{stabilization index} of $\{ J_k \}$.
\end{defn}

The first known families with this property are obtained by taking powers of ideals with homogeneous generators of the same degrees; such ideals are said to be \textbf{equigenerated}.

\begin{thm} [\cite{LV04}, \cite{Singla07}, \cite{Wh14}]\label{thm:stabilization}
Let $I$ be an ideal in a polynomial ring $S= K[x_1, \dots, x_n]$ that is equigenerated in degree $r$.  Then the shape of $\beta(I^k) $ stabilizes.  
\end{thm}

\begin{example}\label{ex:powers}
Let $I = (x^2y+z^3, xyz, yz^2)$.  Macaulay2 (\cite{Macaulay2}) gives the following Betti diagrams of $I^k$ for small $k$.  Note that $\beta_{i,i+j}(I)$ is displayed in the $i$th column and the $j$th row of the Betti diagram.

$$\bordermatrix{ {\beta(I):} & 0 & 1  \cr
					3 & 3 & 1\cr
					4 & - & 1\cr
	}
	\qquad \bordermatrix{ {\beta(I^2):} & 0 & 1  \cr
					6 & 6 & 4\cr
					7 & - & 1\cr
	}
	\qquad  \bordermatrix{ {\beta(I^3):} & 0 & 1 & 2  \cr
					9 & 10 & 9 & 1\cr
					10 & - & 1 & -\cr
	}
	$$
	
$$\bordermatrix{ {\beta(I^4):} & 0 & 1 & 2  \cr
					12 & 15 & 16 & 3\cr
					13 & - & 1 & -
	}
	\qquad \bordermatrix{ {\beta(I^5):} & 0 & 1 & 2 \cr
					15 & 21 & 25 & 6\cr
					16 & - & 1 & -\cr
	}
	\qquad  \bordermatrix{ {\beta(I^6):} & 0 & 1 & 2 \cr
					18 & 28 & 36 & 10\cr
					19 & - & 1 & -
	}.
	$$
For all $k \geq 3$, $\beta(I^k)$ has the following shape
$$\bordermatrix{ {\beta(I^k):} & 0 & 1 & 2  \cr
					3k & a & b & c \cr
					3k+1 & - & d & -
	}$$
	and we see from the examples that the stabilization index of $\{ I^k \}$ is 3.  Looking more closely at the graded Betti numbers appearing in the tables, we see that the entries are given by the following polynomials in $k$:
	$$a = \dfrac{k^2}{2}+ \dfrac{3k}{2}+1,  \quad b = k^2,  \quad c = \dfrac{k^2}{2} - \dfrac{3}{2}k+1,  \quad d = 1.$$
\end{example}

\begin{defn}\label{defn:givenbypolys}
Consider an indexed collection of ideals $\{ J_k \}$. We say that the entries of $ \beta(J_k) $ are \textbf{given by polynomials for $k>>0$} if there exists a number $r$ such that for all $i,j$ there is a number $k_0(i,j)$ and a polynomial $P_{i,j}(k)$ satisfying 
$$\beta_{i,rk+i+j}(J_k) = P_{i,j}(k)$$
for all $k >k_0(i,j)$
\end{defn} 

As suggested by Example \ref{ex:powers}, when $I$ is equigenerated $\beta(I^k) $ has entries given by polynomials.

\begin{thm} [Cor. 2.2 of \cite{Singla07}, Prop. 6.3.6 of \cite{LV04}]\label{thm:polystabilization}
Let $I$ be an ideal in a polynomial ring $S= K[x_1, \dots, x_n]$ that is equigenerated in degree $r$. Then $ \beta(I^k) $ has entries given by polynomials for $k>>0$. Further, the  polynomial functions corresponding to nonzero entries have positive leading coefficients and are of degree less than the analytic spread of $I$, $\ell(I)$.  
\end{thm}

At first glance, it may seem that if $ \beta(J_k) $ has entries given by polynomials for $k>>0$ then   the shape of the $\beta(J_k) $ must stabilize.  However, notice that as $k$ gets large, the number of nonzero entries in $\beta(J_k)$ may keep increasing, and the collection $\{k_0(i,j) \}$ may be unbounded.  The following example illustrates this possibility

\begin{example}
Consider the complete intersection ideal $I=(x^{a}, y^{b}) \subseteq K[x,y]$.  When $1<a \lneq b$, 
\begin{eqnarray*}
\beta_{0, ak+l(b-a)}(I^k) = 1 && \text{ for all } k \geq l \\
\beta_{1, ak+l(b-a)+a}(I^k) = 1 &&\text{ for all } k \geq l
\end{eqnarray*}
and all other Betti numbers are 0 (see \cite{GV05}).  Thus, the Betti numbers of $\beta(I^k)$ are given by polynomials but the shape of the diagram does not stabilize because the number of nonzero entries in $\beta(I^k)$ is unbounded.
\end{example}

The following lemma connects Definitions  and \ref{defn:stabilize} and \ref{defn:givenbypolys}; its proof is immediate from the definitions.

\begin{lem}\label{lem:2conditions}
Let $\{J_k\}$ be an indexed collection of ideals.  Suppose that the entries of $\beta(J_k)$ are given by polynomials in $k$ and that there exists a $k_0$ such that the number of nonzero entries of $\beta(J_k)$ is bounded by a constant for all $k>k_0$.  Then the shape of $ \beta(J_k) $ stabilizes. 
\end{lem}

\begin{question}
\label{question:shapeimplies}
Let $\{ J_k \}$ be a graded system of ideals.  Does stabilization of the shape of $\beta(J_k)$ imply that the entries of $\beta(J_k)$ are given by polynomials? If not, what conditions are necessary?
\end{question}

We now consider the asymptotic behaviour of the graded Betti numbers of other graded systems of ideals. Our main stabilization result concerns graded systems built by taking the generic initial ideals of powers of an equigenerated ideal.  

\begin{thm}\label{thm:mainthm}
Let $I$ be an equigenerated ideal in $S = K[x_1, \dots, x_n]$, where $K$ a field of characteristic 0. Consider the graded reverse lexicographic generic initial system, $\{ \text{gin}(I^k) \}$.  
\begin{enumerate}
\item[a)]  The shape of $ \beta(\text{gin}_{\text{revlex}}(I^k))$ stabilizes.
\item[b)]  The entries of $\beta(\text{gin}_{\text{revlex}}(I^k))$ are given by polynomials for $k>>0$.
\end{enumerate}
\end{thm}

When no monomial order is specified,  $\text{gin}(I)$ represents the generic initial ideal under the reverse lexicographic order and $\text{lexgin}(I)$  represents the generic initial ideal under the lexicographic order.

\begin{example}\label{example:ginstabilization}
Let $I = (x^2y+z^3, xyz, yz^2)$.  Under the reverse lexicographic order, the Betti diagrams of  $\text{gin}(I^k)$ for small values of $k$ are as follows.
$$\bordermatrix{ {\beta(\text{gin}(I)):} & 0 & 1  \cr
					3 & 3 & 2\cr
					4 & 1 & 1\cr
	}
	\qquad \bordermatrix{ {\beta(\text{gin}(I^2)):} & 0 & 1  \cr
					6 & 6 & 5\cr
					7 & 1 & 1\cr
	}
	\qquad  \bordermatrix{ {\beta(\text{gin}(I^3)):} & 0 & 1 & 2  \cr
					9 & 10 & 10 & 1\cr
					10 & 1 & 1 & -\cr
	}
	$$
	
$$\bordermatrix{ {\beta(\text{gin}(I^4)):} & 0 & 1 & 2  \cr
					12 & 15 & 17 & 3\cr
					13 & 1 & 1 & -
	}
	\qquad \bordermatrix{ {\beta(\text{gin}(I^2)):} & 0 & 1 & 2 \cr
					15 & 21 & 26 & 6\cr
					16 & 1 & 1 & -\cr
	}
	\qquad  \bordermatrix{ {\beta(\text{gin}(I^6)):} & 0 & 1 & 2 \cr
					18 & 28 & 37 & 10\cr
					19 & 1 & 1 & -
	}
	$$
	We see that $\beta(\text{gin}(I^k))$ stabilizes to the following shape
$$\bordermatrix{ {\beta(\text{gin}(I^k)):} & 0 & 1 & 2  \cr
					3k & a & b & c \cr
					3k+1 &e & d & -
	}$$
	and its stabilization index is 3.  The non-zero polynomials guaranteed by part (b) of Theorem \ref{thm:mainthm} are:
	$$a = \dfrac{k^2}{2}+ \dfrac{3k}{2}+1, \quad b = k^2+1, \quad c = \dfrac{k^2}{2} - \dfrac{3}{2}k+1, \quad d=e = 1.$$

Under the lexicographic order, the number of nonzero rows in the Betti diagram of $\text{lexgin}(I^k)$ increases as $k$ increases.  For example,

$$\bordermatrix{ {\beta(\text{lexgin}(I)):} & 0 & 1 & 2  \cr
					3 & 3 & 3 & 1\cr
					4 & 2 & 3 & 1\cr
					5 & 1 & 2& 1\cr
					6 & 1 & 1 & -
	}
	\qquad \bordermatrix{ {\beta(\text{lexgin}(I^2)):} & 0 & 1  & 2\cr
					6 & 6  & 8 & 3 \cr
					7 & 4 & 7 & 3\cr
					8 & 3 & 5 & 2\cr
					9 & 2 & 4 & 2\cr
					10 & 2 & 3 & 1\cr
					11 & 1 & 2 & 1\cr
					12 & 1 & 1 & -
	}
	\qquad  \bordermatrix{ {\beta(\text{lexgin}(I^3)):} & 0 & 1 & 2  \cr
					9 & 10 & 15 & 6\cr
					10 & 6 & 10 & 4\cr
					11 & 4 & 8 & 4\cr
					12 & 4 & 7 & 3\cr
					13 & 3 & 6 & 3\cr
					14 & 3 & 5 & 2\cr
					15 & 2 & 4 & 2\cr
					16 & 2 & 3 & 1\cr
					17 & 1 & 2 & 1\cr
					18 & 1 & 1 &-\cr
	}
	$$
We can show that $\text{reg}(\text{lexgin}(I^k)) = 6k$ while the minimal degree generators of $\text{lexgin}(I^k)$ are of degree $3k$.  Therefore, the Betti diagram of $\text{lexgin}(I^k)$ has nonzero entries in $3k+1$ rows, and shape of the Betti diagrams does not stabilize.  However, it appears that for any pair $(i,j)$, $\beta_{i, 3k+i+j}(I^k)$ is a polynomial for $k>>0$, implying a negative answer to Question \ref{question:shapeimplies} for reverse lexicographic generic initial systems.  For example, 
$$\beta_{1, 3k+1+1}(\text{lexgin}(I^k)) = 4k-2$$
for $k \geq 2$. 
\end{example}

While the shape Betti diagrams of the lexicographic generic initial system $\{ \text{lexgin}(I^k) \}$ do not stabilize in the sense of Definition \ref{defn:stabilize}, there appears to be additional uniformity beyond its entries being given by polynomials.  For example, rows 8 through 12 in $\beta(\text{lexgin}(I^2))$ are identical to rows 14 through 18 of $\beta(\text{lexgin}(I^3))$. Similar patterns hold for higher members in the graded system.  Further, while the $I$ above shows that the number of nonzero entries of $\beta(\text{lexgin}(I^k))$ is not always bounded, there are circumstances under which it does, even when $\text{lexgin}(I^k) \neq \text{gin}(I^k)$ (for example, consider $I = (x^2, xy, z^2)$).  Observations like those lead to the following questions.

\begin{question}
Are the entries of $\beta(\text{lexgin}(I^k))$ given by polynomials for all $k>>0$?  Under what conditions does is the number of nonzero entries in $\beta(\text{lexgin}(I^k))$ bounded?  What other types of uniformity do these Betti diagrams exhibit as $k \rightarrow \infty$?
\end{question}


The following example demonstrates that generic coordinates are necessary for Theorem \ref{thm:mainthm}.

\begin{example}
Let $I = (x^2y+z^3, xyz, yz^2)$ and consider the graded system obtained by taking the lexicographic or reverse lexicographic initial ideals of $I^k$, $\{ \text{in}(I^k) \}$.  The regularity of $\text{in}(I^k)$ is increasing at a faster rate than the degree of the minimal generators of $I^k$.  In particular, there are $k+1$ rows in $\beta(\text{in}(I^k))$ with nonzero entries, so the shape of  $\beta(\text{in}(I^k))$ does not stabilize.  

We might hope that the entries in $\beta(\text{in}(I^k))$ are given by polynomials for $k>>0$.  However, if we look more carefully at its entries we see that this does not hold.  Consider, for example,  $\beta_{2, 3k+2}(\text{in}(I^k))$: 

\begin{center}
\begin{tabular}{ l | c | c| c| c| c| c| c| c}
$k$ & 3 & 4 & 5 & 6 & 7 & 8 & 9 & 10\\ 
\hline
$\beta_{2, 3k+2}(\text{in}(I^k))$ & 2 & 4 & 8 & 12 & 18 & 24 & 32 & 40
\end{tabular}  
\end{center}

For all $k \geq 3$, 
$$\beta_{2, 3k+2}(\text{in}(I^k)) = \begin{cases}
0.5k^2-k+0.5 & \text{ for odd $k$}\\
0.5k^2 - k & \text{ for even $k$}
\end{cases}
$$
This function cannot be written as a single polynomial for all $k>>0$.  
\end{example}

As in the lexicographic generic initial system, there appears to be some asymptotic uniformity in the Betti diagrams of the initial system, although it is more complicated than in generic  coordinates.  Similar statements may be made for the Betti diagrams of other graded systems of ideals, including $\{ I^k \}$ when $I$ is not  equigenerated and the system $\{I^{(k)} \}$ composed of the symbolic powers of an ideal.  This leads to the following question:

\begin{question}
In what ways may the Betti diagrams of a graded system of ideals $\{ J_k \}$ display uniformity as $k \rightarrow \infty$?  How can we describe more complicated types of stabilization?  In addition to the systems $\{ I^k \}$ and $\{\text{gin}(I^k)\}$ arising from equigenerated ideals $I$, are there other natural graded systems of ideals for which the shape of $\beta(J_k)$ stabilizes in the sense of Definition \ref{defn:stabilize}?
\end{question}

Bagheri, Chardin, and Ha address similar questions in \cite{BCH11} where they describe the asymptotic behaviour of the graded Betti numbers of multi-graded systems of ideals.

Section \ref{sec:background} of this paper provides background on the resolutions of generic initial ideals and Section \ref{sec:proof} contains the proof of Theorem \ref{thm:mainthm}. Additional corollaries of Theorem \ref{thm:mainthm} are included in Section \ref{subsec:consequences}, including a description of the Boij-S\"{o}derberg decompositions of $\beta(\text{gin}(I^k))$.

%% file: Background.tex
\section{Background: Graded Betti Numbers of Generic Initial Ideals}
\label{sec:background}

In this section we record facts related to the graded Betti numbers of generic initial ideals.  Throughout, $S = K[x_1, \dots, x_n]$ is a polynomial ring with the standard grading over a field of characteristic zero and $I$ is a homogeneous ideal in $S$.  We use $\text{gin}_{\tau}(I)$ to  denote the generic initial ideal under the term order $\tau$ and $\text{gin}(I)$ to denote the generic initial ideal under the reverse lexicographic order. 

To describe the graded minimal free resolution of $\text{gin}(I^k)$, for any monomial $m$, we will denote
$$\text{max}(m) = \text{max} \{i : x_i \text{ divides } m \}.$$

\begin{defn}[See Construction 28.6 in \cite{Peeva10}]
Let $M$ be a Borel-fixed ideal minimally generated by monomials $m_1, \dots, m_{\ell}$.  For each $m_i$ and for each sequence of natural numbers $1 \leq j_1 < \cdots < j_p < \text{max}(m_i)$, let $S(m_i;x_{j_1} \cdots x_{j_p})$ denote a free $S$-module with one generator in degree $\text{deg}(m_i) + p$. The basis of the \textbf{Eliahou-Kervaire resolution} of $M$ is given by 
 $$\mathcal{B} = \{ S(m_i; j_1, \dots, j_p) : 1 \leq j_1 < \cdots < j_p < \text{max}(m_i), 1 \leq i \leq r \}$$  
 where  $S(m_ix_{j_1} \cdots x_{j_p})$ is in homological degree $p$.. 
\end{defn}

Recalling that $\text{gin}_{\tau}(I)$ is Borel-fixed (\cite{Galligo74}), we obtain the following result.

\begin{thm}[\cite{EK90}]\label{thm:EKres}
For any  homogeneous ideal $I \subseteq S$,  the Eliahou-Kervaire resolution produces a minimal free resolution of $\text{gin}_{\tau}(I)$. 
\end{thm}

 
\begin{cor}\label{cor:EKdegs}
Let $M$ be a Borel-fixed ideal minimally generated by monomials $m_1, \dots, m_{\ell}$.  Then the graded Betti numbers of $M$ are given by 
$$\beta_{p, p+q}(M) = \sum_{\text{deg}(m_i) = q, 1 \leq i \leq \ell} {\text{max}(m_i) - 1 \choose p}$$
\end{cor}

The following definition and theorem say that the Betti diagrams of an ideal and any of its generic initial ideals are related by cancelling numbers along diagonals.  

\begin{defn}[see \cite{Peeva10}]
We say that a sequence of numbers $\{ q_{i,j} \}$ is obtained from a sequence of numbers $\{ p_{i,j} \}$ by a \textbf{consecutive cancellation} if there are indices $s$ and $r$ such that 
\begin{eqnarray*}
q_{s, r} &=& p_{s,r} - 1, \quad q_{s+1, r} = p_{s+1, r}-1\\
q_{i,j} &=& p_{i,j} \text{ for all other values of } i,j
\end{eqnarray*}
\end{defn}

\begin{thm}[\cite{Green98}]
\label{thm:consecutivecancellation}
Let $I$ be a homogeneous ideal in $S$.  For any term order $\tau$, the graded Betti numbers $\beta_{i,j}(I)$ may be obtained from the graded Betti numbers $\beta_{i,j}(\text{gin}_{\tau}(I))$ by a sequence of consecutive cancellations.   
\end{thm}  

\begin{example}\label{example:consecutivecancellation}
Let $I = (x^2y+z^3, xyz, yz^2)$. The Betti diagram of $\text{gin}(I^3)$ is 

$$\bordermatrix{ ~ & 0 & 1 & 2   \cr
					9 & 10 & 10 & 1 \cr
					10 & 1 & 1& -\cr 
	}.$$
Theorem \ref{thm:consecutivecancellation} says that $\beta(I^3)$ will be obtained by making consecutive cancellations.  There are two possible consecutive cancellations in this case:  cancelling $\beta_{0, 10}(\text{gin}(I^3))$ with $\beta_{1, 10}(\text{gin}(I^3))$ and cancelling $\beta_{1, 11}(\text{gin}(I^3))$ with $\beta_{2,11}(\text{gin}(I^3))$.  Using Macaulay2, we find that only the first consecutive cancellation is needed and $\beta(I^3)$ is given by
 $$\bordermatrix{ ~ & 0 & 1 & 2   \cr
					9 & 10 & 9 & 1 \cr
					10 & - & 1& -\cr 
	}.$$
	\end{example}
This example demonstrates that not all possible consecutive cancellations need to occur when we pass from the Betti digram of $\text{gin}(I)$ to the Betti diagram of $I$. In general, there are no techniques for determining which consecutive cancellations will occur.

\begin{defn}
A Betti number $\beta_{i,i+j}(I) \neq 0$ is called \textbf{extremal} if $\beta_{k,k+l}(I) = 0$ for all pairs $(k,l) \neq (i,j)$ with $k\geq i$ and $l \geq j$.  An extremal Betti number is the upper left-hand corner of a block of zeros within a Betti diagram.
\end{defn}

Notice that if $I$ is an equigenerated ideal, the extremal Betti numbers of $I^k$ occur in the same relative positions for all $k>>0$ by Theorem \ref{thm:stabilization}.  Our interest in extremal Betti numbers comes as a result of the following theorem, which says that the reverse lexicographic generic initial ideal preserves extremal Betti numbers.

\begin{thm}[Theorem 4.3.17 of \cite{BCP99}]\label{thm:extremalbetti} 
Let $I$ be a homogeneous ideal of $S$.  Then for any $i, j \subset \mathbb{N}$, 
\begin{enumerate}
\item[(a)]  $\beta_{i,i+j}(I)$ is extremal if and only if $\beta_{i,i+j}(\text{gin}(I))$ is extremal;
\item[(b)]  if $\beta_{i,i+j}(I)$ is extremal, then $\beta_{i,i+j}(I) = \beta_{i,i+j}(\text{gin}(I))$.
\end{enumerate}
\end{thm}

\begin{cor}[\cite{BS87}]\label{cor:regpreserved}
Under the reverse lexicographic order,
$$\text{reg}(I) = \text{reg}(\text{gin}(I)).$$
\end{cor}

Another advantage of the reverse lexicographic order is that hyperplane sections interact nicely with generic initial ideals under this order.  Consider a linear form $h = \sum_{j=1}^{n} h_j x_j \in S_1$ such that $h_n \neq 0$.  The homomorphism $\phi: S = K[x_1, \dots, x_n] \rightarrow \hat{S} = K[x_1, \dots, x_{n-1}]$ defined by $\phi(x_j) = x_j$ for $j \neq n$ and $\phi(x_n) = -\dfrac{1}{h_n}(\sum_{j \neq n} h_j x_j)$ defines an isomorphism between $S/(h)$ and $\hat{S}$ that preserves the degree of polynomials not in the kernel.  

\begin{defn}
Given an ideal $I$ in $S$, the \textbf{$h$-hyperplane section} of $I$ is the ideal $\phi(I)$ in $\hat{S}$.  It is  denoted $I_h$.   
\end{defn}

\begin{thm}[\cite{BS87}, \cite{Green98}]\label{thm:hyperplanerestriction}  
Let $I$ be a homogeneous ideal in $S$ and let $h \in S_1$ be a generic linear form.  Then 
$$\text{gin}(I_h) = \text{gin}(I) |_{x_n \rightarrow 0}.$$
\end{thm}

%% file: StabilizationProof.tex
\section{Stabilization of $\beta(\text{gin}_{}(I^k))$}
\label{sec:proof}

The goal of this section is to prove Theorem \ref{thm:mainthm} describing the stabilization of $\beta(\text{gin}(I^k))$.  Throughout, $I$ is a homogeneous ideal of $S=K[x_1, \dots, x_n]$ with the standard grading.  We denote the generic initial ideal of $I$ with respect to the reverse lexicographic order by $\text{gin}(I)$.

\subsection{Part (a) of Theorem \ref{thm:mainthm} follows from part (b)}

The following lemma claims that for all $k>>0$ the number of nonzero entries in $\beta(\text{gin}(I^k))$ is bounded.  Note that it only holds for the reverse lexicographic order (see Example \ref{example:ginstabilization}).
\begin{lem}\label{lem:ginwindow}
Suppose that $I$ is an ideal equigenerated in degree $r$, and let $c$ and $k_0$ be the constants such that $\text{reg}(I^k) = rk+c$ for $k>k_0$, guaranteed by Theorem \ref{thm:stabilization} .  Then for all $k>k_0$, the nonzero entries of $\beta(\text{gin}(I^k))$ are confined to the following  window.
$$\bordermatrix{ ~ & 0 & 1 & \cdots & n-1 \cr
					rk & \ast & \ast & \cdots & \ast \cr
					rk+1 &  \ast & \ast & \cdots & \ast \cr
					\vdots & \vdots & \vdots & \vdots & \vdots \cr
					rk+c &  \ast & \ast & \cdots & \ast \cr
	}$$
\end{lem}

\begin{proof}
Fix $k>k_0$.  Since all elements of $\text{gin}(I^k)$ are of degree at least $rk$, the least row of $\beta(\text{gin}(I^k))$  containing nonzero entries is the $rk$th row.  By Corollary \ref{cor:regpreserved}, $\text{reg}(I^k) = \text{reg}(\text{gin}(I^k)) = rk+c$ so the $(rk+c)$th row of $\beta(\text{gin}(I^k))$ is the greatest containing nonzero entries.  Finally, since $\text{projdim}(\text{gin}(I^k)) \leq n-1$, the $(n-1)$st column is the greatest column containing nonzero entries.
\end{proof}

Since Lemma \ref{lem:ginwindow} says that only a finite number of $\beta_{i,i+rk+j}(\text{gin}(I^k))$ may be nonzero, part (a) of Theorem \ref{thm:mainthm} follows from part (b) of the theorem and Lemma \ref{lem:2conditions}.  Therefore, we will focus on the proof of part (b).  

\subsection{Proof of Part (b) of Theorem \ref{thm:mainthm}}

\begin{proof}
We will proceed by induction on the number of variables $n$ in $S$.  

\textbf{Base Case: Result holds for 2 variables.}

Suppose that $I \subseteq K[x, y]$ is equigenerated in degree $r$.  Then by Theorem \ref{thm:polystabilization},  $\beta_{i, rk+i+j}(I^k)$ is given by a polynomial $P_{i,j}(k)$ for all $k>>0$.  By consecutive cancellation (Theorem \ref{thm:consecutivecancellation}), there exist integers $c_{i,j}(k)$ such that $\beta_{i,rk+i+j}(\text{gin}_{}(I^k)) =  P_{i,i+j}(k) + c_{i,i+j}(k)$. 

Note that $\text{projdim}(\text{gin}_{}(I^k))$ and  $\text{projdim}(I^k)$ are at most 2 for all $k>>0$, so  $c_{i, i+j}(k) = P_{i,i+j}(k)=0$ whenever $i \geq 2$ .   

$$\bordermatrix{ \underline{\beta(\text{gin}_{}(I^k))} & 0 & 1  \cr
					rk & P_{0,0}(k)+ c_{0,0}(k) & P_{1,0}(k) + c_{1,0}(k)\cr
					rk+1 & P_{0,1}(k)+ c_{0,1}(k) & P_{1,1}(k) + c_{1,1}(k)\cr
					rk+2 & P_{0,2}(k)+ c_{0,2}(k) & P_{1,2}(k) + c_{1,2}(k)\cr
										~ & \vdots & \vdots \cr
	}$$
				
We will now proceed by induction on the row index $j$ to show that $c_{i,j}(k)$ is a polynomial in $k$ for all $k>>0$.  This will imply that the entries in $\beta(\text{gin}_{}(I^k))$ are given by polynomials for all $k>>0$.
					
According to Lemma \ref{lem:ginwindow}, $\beta_{i, rk}(\text{gin}_{}(I^k)) =0$ for all $i>0$.  By consecutive cancellation, Theorem \ref{thm:consecutivecancellation}, this implies that  $c_{0,0}(k) = 0$.  The Eliahou-Kervaire resolution implies that, for all $k>>0$, $P_{0,0}(k)$ gives the number of minimal generators of $\text{gin}_{}(I^k)$ of degree $rk$, while $P_{1,0}(k) + c_{1,0}(k)$ gives the number of minimal generators of $\text{gin}_{}(I^k)$ of degree $rk$ divisible by $y$.  We know that $x^{rk}$ the only  generator of $\text{gin}_{}(I^k)$ of degree $rk$ not divisible by $y$, so $P_{1,0}(k) + c_{1,0}(k) = P_{0,0}(k) - 1$.  Therefore $c_{1,0}(k) = P_{0,0}(k) -1-P_{1,0}(k)$ is a polynomial in $k$ for $k>>0$.  

Suppose that we know $c_{i,j}(k)$ is a polynomial in $k$ for  $i = 0, 1$ and all  $j < J$.  First notice that both  $P_{0,J}(k) + c_{0,J}(k)$ and  $P_{1,J}(k) + c_{1,J}(k)$ represent the number of generators of $\text{gin}_{}(I^k)$ of degree $rk+J$ since $y$ divides all generators of degree greater than $rk$.  Thus, $P_{0,J}(k) + c_{0,J}(k) = P_{1,J}(k) + c_{1,J}(k)$ and $c_{1,J}(k)$ is a polynomial for all $k>>0$ if $c_{0,J}(k)$ is.  However, by consecutive cancellation, $c_{0,J}(k) = c_{1, J-1}(k)$ is polynomial by our inductive assumption.  
Therefore, Theorem \ref{thm:mainthm} holds when $I$ is an equigenerated ideal in $K[x,y]$.

\textbf{Inductive Step:  Hyperplane Restriction}

Now suppose that Theorem \ref{thm:mainthm} holds for  all equigenerated ideals in $ K[x_1, \dots, x_{n-1}]$.  We will show that this implies the theorem  for equigenerated ideals in $K[x_1, \dots, x_n]$. The key element of this argument is hyperplane restriction, Theorem \ref{thm:hyperplanerestriction}.  Let $I$ be an equigenerated ideal in $K[x_1, \dots, x_n]$ and let $I_h$ be a generic hyperplane section of $I$.    $I_h$ is a homogeneous ideal of $K[x_1, \dots, x_{n-1}]$ equigenerated in degree $r$ and 
$$\text{gin}_{}((I_h)^k) = \text{gin}_{}((I^k)_h) = \text{gin}_{}(I^k)|_{x_n \rightarrow 0}.$$

Since our inductive assumption applies to $\{ I_h \} \subseteq K[x_1, \dots, x_{n-1}]$, there are polynomials $f_{i,j}(k)$ such that $\beta_{i, rk+i+j}(\text{gin}_{}((I_h)^k)) =\beta_{i, rk+i+j}(\text{gin}_{}((I^k)_h))  = f_{i,j}(k)$ for $i$, $j$, and $k>>0$.  From  Theorem \ref{thm:polystabilization}, we know that there are polynomials $P_{i,j}(k)$ such that $\beta_{i,rk+i+j}(I^k) = P_{i,j}(k)$ for all $k>>0$ and by consecutive cancellation $\beta_{i,rk+i+j}(\text{gin}_{}(I^k)) = P_{i,j}(k) + c_{i,j}(k)$ for some numbers $c_{i,j}(k)$.  To complete the inductive step it is sufficient to show that the $c_{i,j}(k)$ are polynomials in $k$ for all $k>>0$.  

$$\bordermatrix{ \underline{\beta(\text{gin}_{}(I^k)_h):} & 0 & 1  & \cdots & n-2 \cr
					rk & f_{0,0}(k)& f_{1,0}(k) & \cdots & f_{n-2, 0}(k) \cr
					rk+1 & f_{0,1}(k) & f_{1,1}(k) &  \cdots & f_{n-2, 1}(k)  \cr
					\vdots & \vdots & \vdots & \cdots & \vdots \cr
					rk+\rho' & f_{0,\rho'}(k)& f_{1,\rho'}(k)& \cdots & f_{n-2, \rho'}(k)  \cr
	}$$

$$\bordermatrix{ \underline{\beta(\text{gin}_{}(I^k)):} & 0 & 1  & \cdots & n-1 \cr
					rk & P_{0,0}(k)+ c_{0,0}(k) & P_{1,0}(k) + c_{1,0}(k) & \cdots & P_{n-1, 0}(k) + c_{n-1,0}(k)\cr
					rk+1 & P_{0,1}(k)+ c_{0,1}(k) & P_{1,1}(k) + c_{1,1}(k) &  \cdots & P_{n-1, 1}(k) + c_{n-1,1}(k) \cr
					\vdots & \vdots & \vdots & \cdots & \vdots \cr
					rk+\rho & P_{0,\rho}(k)+ c_{0,\rho}(k) & P_{1,\rho}(k) + c_{1,\rho}(k) & \cdots & P_{n-1, \rho}(k) + c_{n-1, \rho}(k) \cr
	}$$


By Theorem \ref{thm:hyperplanerestriction}, for any  $\delta$ the set of degree $\delta$ minimal generators of $\text{gin}_{}(I^k)$ not divisible by $x_n$ is the same as the set of degree $\delta$ minimal generators of  $\text{gin}_{}((I^k)_h)$.  We will use this fact along with the Eliahou-Kervaire resolution to determine how the $f_{i,j}$s and $P_{i,j}$s are related, row-by-row.

For a monomial ideal $M$, let $m_{i, \delta}(M)$ denote the number of minimal generators of $M$ of degree $\delta$ divisible by $x_i$ and no larger-indexed variable.  $\beta_{0, rk}(M)$ gives the number of generators of $M$ of degree $rk$. Since $\beta_{i, rk}(\text{gin}_{}(I^k)) = 0$ for all $i>0$, $c_{0,0}(k) = 0$ by consecutive cancellation.  Then 
\begin{eqnarray*}
m_{n, rk}(\text{gin}_ {}(I^k)) &=& \beta_{0, rk}(\text{gin}_{}(I^k)) - \beta_{0, rk}(\text{gin}_{}((I^k)_h))\\
&=& P_{0,0}(k) - f_{0,0}(k).
\end{eqnarray*}
We know that $P_{0,0}(k)$ and $f_{0,0}(k)$ are eventually polynomial in $k$, so $m_{n, rk}(\text{gin}_{}(I^k))$ is a polynomial in $k$ for all $k>>0$.

Similarly, in homological degree $I$, Corollary \ref{cor:EKdegs} implies that
\begin{eqnarray*}
 [P_{I,0}(k) +c_{I,0}(k)] - f_{I,0}(k) &=&  \sum_{q=I+1}^{n} {q-1 \choose I} m_{q, rk}(\text{gin}_{}(I^k)) - \sum_{q=I+1}^{n-1} {q-1 \choose I} m_{q, rk}(\text{gin}_{}((I^k)_h))\\
 &=& {n-1 \choose I} m_{n, rk}(\text{gin}_{}(I^k))
 \end{eqnarray*}
 Since we know that the $P_{I,0}(k)$, $f_{I,0}(k)$, and $ m_{n, rk}(\text{gin}_{}(I^k))$ are polynomial for all $k>>0$, $c_{I, 0}(k)$ must be polynomial for $k>>0$ as well.  Therefore, we know that the cancellation numbers $c_{i, 0}(k)$ in the first nonzero row of $\beta(\text{gin}_{}(I^k))$ are polynomial for all $k>>0$.  
 
 Now assume that we know that the cancellation numbers $c_{i, j}(k)$ are polynomial for all $k>>0$ whenever $j < J$.  By consecutive cancellation, 
 $$c_{0,J}(k) =c_{1, J-1}(k) - c_{2, J-2}(k) + \cdots + (-1)^n c_{n-1, J-(n-1)}(k).$$
 By our assumption, this implies that $c_{0, J}(k)$ is also polynomial for all $k>>0$.  
 As before, 
 $$[P_{0, J}(k) + c_{0, J}(k)] - f_{0, J}(k) = m_{n, rk+J}(\text{gin}(I^k))$$
 so $m_{n, rk+J}(\text{gin}_{}(I^k))$ is polynomial for all $k>>0$ as well.  
 
 In column $I$ and row $rk+J$, 
\begin{eqnarray*}
 [P_{I, J}(k) +c_{I,J}(k)] - f_{I,J}(k) &=&  \sum_{q=I+1}^{n} {q-1 \choose I} m_{q, rk+J}(\text{gin}_{}(I^k)) - \sum_{q=I+1}^{n-1} {q-1 \choose I} m_{q, rk+J}(\text{gin}_{}((I^k)_h))\\
 &=& {n-1 \choose I} m_{n, rk+J}(\text{gin}_{}(I^k))
 \end{eqnarray*}
 Since $m_{n, rk+J}(\text{gin}_{}(I^k))$ is polynomial for all $k>>0$, we see that $c_{I,J}(k)$ must also be polynomial in $k$ for all $k>>0$. Therefore, all cancellation numbers in row $rk+J$ are polynomial for $k>>0$. We  conclude by induction on the row index that that all entries in the table $\beta(\text{gin}_{}(I^k))$ are given by polynomials in $k$ for $k>>0$ when $I \subseteq K[x_1, \dots, x_n]$ is equigenerated.  
 
 Therefore, by induction on the number of variables, we conclude that the theorem holds for all equigenerated ideals $I$ in any number of variables.

\end{proof}

\subsection{Consequences of the Proof}
\label{subsec:consequences}

By Theorem \ref{thm:polystabilization}, the polynomials $P_{i,j}(k)$ that give the entries of $\beta(I^k)$ are of degree less than $\ell(I)$, the analytic spread of $I$.  If we follow the degrees of the polynomials in the above proof, we see that the same can be said for the functions giving the cancellation numbers.  
\begin{cor}
Let $I$ be an ideal equigenerated in degree $r$.  Then the degree of each polynomial $Q_{i,j}(k) = \beta_{i,i+j+rk}(\text{gin}_{}(I^k))$ guaranteed by Theorem \ref{thm:mainthm} is less than $\ell(I)$. 
\end{cor}

Similarly, we can follow the stabilization indices in the proof to obtain the following result.

\begin{cor}
Let $I$ be an equigenerated ideal with  generators of degree $r$.  Then the stabilization index of $\{ \text{gin}(I^k) \}$ is at most the stabilization index of $\{ I^k \}$.  
\end{cor}

We may also consider the structure of the Boij-S\"{o}derberg decompositions of $\beta(\text{gin}(I^k))$. The proof of the main result from \cite{MT15} holds for any family of ideals $\{J_k\}$ such that the shape of $\beta(J_k)$ stabilizes and whose entries are given by polynomials in $k$.  Therefore, Theorem \ref{thm:mainthm} gives the following result.  

\begin{cor}\label{cor:bsdecompos}
Let $I$ be an equigenerated ideal with  generators of degree $r$.  Then there are integers $m$ and $K$ such that for all $k> K$, the positive Boij-S\"{o}derberg decomposition of $\beta(\text{gin}(I^k))$ is of the form
$$\beta(\text{gin}(I^k)) = w_1(k)\pi_1(k)+\cdots + w_m(k)\pi_m(k)$$
where each $w_i(k)$ is a polynomial in $k$ with rational coefficients and $\pi_1(k) < \pi_2(k) < \cdots < \pi_m(k)$ is a translated family of chains.  
\end{cor}